\documentclass[a4paper,12pt]{article}
\usepackage[top=2.5cm,bottom=2.5cm,left=2.5cm,right=2.5cm]{geometry}
\usepackage{cite, amsmath, amssymb,color}
\newtheorem{theo}{Theorem}
\newtheorem{theorem}{Theorem}
\newtheorem{lemma}[theorem]{Lemma}

\newtheorem{corollary}[theo]{Corollary}

\newenvironment{proof}{\par \noindent \textbf{Proof: }}{\QED \par \bigskip \par}
\newcommand{\QED}{\hfill$\square$}

\allowdisplaybreaks[4]

\begin{document}

\baselineskip=0.30in

\vspace*{10mm}

\begin{center}
{\Large \bf \boldmath On the Vertex-Degree-Function Indices of Connected\\[2mm] $(n,m)$-Graphs of Maximum Degree at Most Four}

\vspace{8mm}

{\large \bf Abeer M. Albalahi$^{1}$,  Igor \v{Z}. Milovanovi\'c$^{2}$, Zahid Raza$^{3}$,\\ Akbar Ali$^{1,}\footnote{Corresponding author}$,
Amjad E. Hamza$^{1}$}

\vspace{6mm}

\baselineskip=0.23in

$^1${\it Department of Mathematics, College of Science,\\ University of Ha\!'il, Ha\!'il, Saudi Arabia}\\
{\tt a.albalahi@uoh.edu.sa, akbarali.maths@gmail.com, aboaljod2@hotmail.com} \\[3mm]
$^2${\it Faculty of Electronic Engineering,\\ University of Ni\v{s}, Ni\v{s}, Serbia}\\
{\tt igor.milovanovic@elfak.ni.ac.rs}\\[3mm]
$^3${\it Department of Mathematics, College of Sciences,\\ University of Sharjah, Sharjah, UAE}\\
{\tt zraza@sharjah.ac.ae}

\makeatletter

\def\@makefnmark{}

\makeatother


\vspace{10mm}

\baselineskip=0.23in

{\bf Abstract}

\end{center}
\noindent
Consider a graph $G$ and a real-valued function $f$ defined on the degree set of $G$. The sum of the outputs $f(d_v)$ over all vertices $v\in V(G)$ of $G$ is usually known as the vertex-degree-function indices and is denoted by $H_f(G)$, where $d_v$ represents the degree of a vertex $v$ of $G$. This paper gives sharp bounds on the index $H_f(G)$ in terms of order and size of $G$ when $G$ is connected and has the maximum degree at most $4$. All the graphs achieving the derived bounds are also determined by their degree sets. Bounds involving several existing indices -- including the general zeroth-order Randi\'c index and coindex,  the general multiplicative first/second Zagreb index, the variable sum lodeg index, the variable sum exdeg index -- are deduced as the special cases of the obtained ones.\\[2mm]

\baselineskip=0.40in

\section{Introduction}

This study is concerned with only connected and finite graphs. The (chemical-)graph theoretical concepts used in the this paper without providing their definitions can be found in the related books like \cite{Wagner-18,Trinajsti-92,Chartrand-16,Bondy08}.

A topological index is a function defined on the set of all graphs with the condition that it remains the same under the graph isomorphism.
The degree set of a graph $G$ is the set consisting of all distinct elements of the degree sequence of $G$. Consider a graph $G$ and a real-valued function $f$ defined on the degree set of $G$. The sum of the outputs $f(d_v)$ over all vertices $v\in V(G)$ of $G$ is usually known as the vertex-degree-function indices and is denoted by $H_f(G)$, where $d_v$ represents the degree of a vertex $v$ of $G$. Thus,
\begin{equation}\label{99juyhhh}
H_f(G) = \sum_{v\in V(G)} f(d_v).
\end{equation}
Although the terminology and notation of the topological index $H_f$ that is being used by several researchers was coined by Yao et al. \cite{MH2}, to the best of authors' knowledge such indices were studied first by Linial and Rozenman in \cite{Linial-02}.  These indices have been the subject of several recent papers; see for example the recent articles \cite{Tomescu-MATCH-22,Hu-MATCH-22-a,Tomescu-DAM-22}, recent review paper \cite{Li-DML-22}, and related publications cited therein.

If vertices $u$ and $v$ are adjacent in $G$, we write $u\sim v$, otherwise we write $u\nsim v$. Let $TI(G)$ be a vertex--degree--based topological index of the form:
$$
TI(G) =\sum_{u\sim v}(f(d_u)+f(d_v))= \sum_{u\in V(G)}d_uf(d_u)\,;
$$
the right-handed identity is a special case of a more general identity reported in \cite{c18}.
Then, the  corresponding coindex, $\overline{TI}(G)$ can be defined \cite{c19,c20} as:
$$
\overline{TI}(G)= \sum_{u\nsim v} (f(d_u)+f(d_v))= \sum_{u\in V(G)} (n-1-d_u)f(d_u)\,.
$$
The following identity is valid
\begin{equation}\label{l10a}
  TI(G)+\overline{TI}(G) = (n-1)\sum_{u\in V(G)}f(d_u)= (n-1)H_f(G)\,.
\end{equation}

In what follows, some existing indices are given that are special cases of Equation \eqref{99juyhhh}.
\begin{itemize}
  \item Equation \eqref{99juyhhh} gives the general zeroth-order Randi\'c index if $f(x)=x^\alpha$ (see for example \cite{Milo-20,Ali19,c21,c22,c23}), where $\alpha$ is a real number.

  \item The general zeroth--order Randi\' c coindex is obtained from equation \eqref{99juyhhh} corresponding to the choice $f(x)=(n-1-x)x^{\alpha -1}$, where $n$ is the order of the graph under consideration and $\alpha$ is a real number (see e.g. \cite{c24,c25}). Particularly, if $\alpha=3$, then the forgotten topological coindex $\overline{F}(G)=\sum_{u\in V(G)}(n-1-d_u)d_u^2$ is obtained (see for example \cite{c14, c26}); the forgotten topological coindex is same as the Lanzhou index \cite{Vuki-18}.

  \item One gets the natural logarithm of the general multiplicative first Zagreb index (general multiplicative second Zagreb index, respectively) \cite{Vetrik-18} by taking $f(x)=\ln  x^a$  ($f(x)=\ln  x^{ax}$, respectively), where $a\in \mathbb{R}$ (that is the set of real numbers).
  \item The substitution $f(x)= x(\ln x)^a$ in Equation \eqref{99juyhhh} yields the variable sum lodeg index \cite{Vuki-11}, where $a\in \mathbb{R}$
  \item If $f(x)=xa^x$ then Equation \eqref{99juyhhh} gives the variable sum exdeg index (see for example \cite{Vuki-11,Ali-DAM-18}), where $a>0$ with $a\ne1$.
  \end{itemize}

A graph with $n$ vertices and $m$ edges is known as an $(n,m)$-graph. A chemical graph is the one with the maximum degree at most four. This paper gives sharp bounds on the index $H_f(G)$ for chemical $(n,m)$-graphs in terms of $m$ and $n$. All the graphs achieving the derived bounds are also identified. Bounds involving the above-mentioned existing indices (that is, the general zeroth-order Randi\'c index and coindex,  the general multiplicative first/second Zagreb index, the variable sum lodeg index, the variable sum exdeg index) are deduced as the special cases of the obtained bounds.

\section{Main Results}

For a graph $G$, its number of vertices having the degree $r$ is denoted by $n_r$.
If $G$ is a chemical $(n,m)$-graph, then
\begin{equation}\label{MT-Eq-1}
H_f(G)=\sum_{i = 1}^4   n_{i}\, f(i),
\end{equation}
\begin{equation}\label{MT-Eq-2}
\sum_{i = 1}^4 n_i=n,
\end{equation}
\begin{equation}\label{MT-Eq-3}
\sum_{i = 1}^4 i\,n_i=2m,
\end{equation}
We solve Equations \eqref{MT-Eq-2} and \eqref{MT-Eq-3} for the quantities $n_{1},n_4$, and then substitute their values
in Equation (\ref{MT-Eq-1}):
\begin{align}\label{MT-Eq-5}
H_f(G)&=
\frac{1}{3}\Big(4f(1) - f(4)\Big)n + \frac{2}{3}\Big(f(4) - f(1)\Big)m\nonumber\\[4mm]
&\quad \ + \left(f(2) -\frac{2}{3}f(1) -\frac{1}{3}f(4)\right) n_2 + \left(f(3) -\frac{1}{3}f(1) -\frac{2}{3}f(4)\right) n_3 \,.
\end{align}
Let us take
\begin{align}\label{MT-Eq-6}
\Gamma_{\!\! f}(G)&= \left(f(2) -\frac{2}{3}f(1) -\frac{1}{3}f(4)\right) n_2 + \left(f(3) -\frac{1}{3}f(1) -\frac{2}{3}f(4)\right) n_3\,.
\end{align}
Now, Equation (\ref{MT-Eq-5}) yields
\begin{equation}\label{MT-Eq-7}
H_f(G)= \frac{1}{3}\Big(4f(1) - f(4)\Big)n + \frac{2}{3}\Big(f(4) - f(1)\Big)m
+\Gamma_{\!\! f}(G)\,.
\end{equation}
Let
\begin{equation}\label{MT-Eq-rev-1}
\xi_{1}=f(2)-\frac{2}{3} f(1)-\frac{1}{3} f(4) \quad \text{and} \quad \xi_{2}=f(3)-\frac{1}{3} f(1)-\frac{2}{3} f(4)
\end{equation}
be the coefficients of $n_{2}$ and $n_{3}$, respectively, in \eqref{MT-Eq-6}. From Equation \eqref{MT-Eq-7}, it is evident that if one wants to establish a bound on $H_f$ for chemical $(n,m)$-graphs in terms of $m$ and $n$, it is enough to determine the least or greatest $\Gamma_{\!\! f}$-value for chemical $(n,m)$-graphs. Thence, in the next lemma, we derive a bound on $\Gamma_{\!\! f}$ for chemical $(n,m)$-graphs.
\begin{lemma}\label{MT-Lem-1}
Let $G$ be a chemical $(n,m)$-graph such that $n_2+n_3\ge2$.
\begin{description}
  \item[(i).] If both $\xi_1$ and $\xi_2$ are negative such that $2\xi_2 < \xi_1 < \xi_2/2$, then
$$
\Gamma_{\!\! f}(G) < \min\left\{\xi_1,\xi_2\right\}.
$$
  \item[(ii).] If both $\xi_1$ and $\xi_2$ are positive such that $\xi_2/2 < \xi_1 < 2\xi_2$, then
$$
\Gamma_{\!\! f}(G) > \max\left\{\xi_1,\xi_2\right\}.
$$
\end{description}

\end{lemma}

\begin{proof}
(i) Take $\max\left\{\xi_1,\xi_2\right\}=\xi_{max}$. Note that
\[
\Gamma_{\!\! f}(G)= \xi_1 n_2 + \xi_2 n_3 \le (n_2 + n_3)\xi_{max}  \le 2\xi_{max}<\min\left\{\xi_1,\xi_2\right\}.
\]
(ii) Let $\min\left\{\xi_1,\xi_2\right\}=\xi_{min}$. Then
\[
\Gamma_{\!\! f}(G)= \xi_1 n_2 + \xi_2 n_3 \ge (n_2 + n_3)\xi_{min}  \ge 2\xi_{min}>\max\left\{\xi_1,\xi_2\right\}.
\]

\end{proof}

Recall that the degree set of a graph $G$ is the set of all unequal degrees of vertices of $G$.

\begin{theorem}\label{thm-main-01}
Let\, $G$ be a chemical $(n,m)$-graph, where $n\ge5$. Let $\xi_1$ and $\xi_2$ be the numbers defined in \eqref{MT-Eq-rev-1}.
\begin{description}
  \item[(i).] If both $\xi_1$ and $\xi_2$ are negative such that $2\xi_2 < \xi_1 < \xi_2/2$, then
\begin{align*}
H_f(G)&\le \frac{1}{3}\Big(4f(1) - f(4)\Big)n + \frac{2}{3}\Big(f(4) - f(1)\Big)m\nonumber\\[4mm]
& \quad \ +
\begin{cases}
\displaystyle f(2) -\frac{2}{3}f(1) -\frac{1}{3}f(4)       & \text{if~ $2m-n\equiv1\pmod{3}$}\\[6mm]
\displaystyle f(3) -\frac{1}{3}f(1) -\frac{2}{3}f(4)       & \text{if~ $2m-n\equiv2\pmod{3}$}\\[6mm]
0                                                          & \text{if~ $2m-n\equiv0\pmod{3}$}
\end{cases}
\end{align*}
with equality if and only if the degree set of $G$ is \\
$\bullet$ $\{1,2,4\}$ and $G$ contains only one vertex of degree $2$ whenever\, $2m-n\equiv1\pmod{3}$;\\
$\bullet$ $\{1,3,4\}$ and $G$ contains only one vertex of degree $3$ whenever\, $2m-n\equiv2\pmod{3}$;\\
$\bullet$ $\{1,4\}$ whenever\, $2m-n\equiv0\pmod{3}$.

  \item[(ii)] If both $\xi_1$ and $\xi_2$ are positive such that $\xi_2/2 < \xi_1 < 2\xi_2$, then
\begin{align*}
H_f(G)&\ge \frac{1}{3}\Big(4f(1) - f(4)\Big)n + \frac{2}{3}\Big(f(4) - f(1)\Big)m\nonumber\\[4mm]
& \quad \ +
\begin{cases}
\displaystyle f(2) -\frac{2}{3}f(1) -\frac{1}{3}f(4)       & \text{if~ $2m-n\equiv1\pmod{3}$}\\[6mm]
\displaystyle f(3) -\frac{1}{3}f(1) -\frac{2}{3}f(4)       & \text{if~ $2m-n\equiv2\pmod{3}$}\\[6mm]
0                                                          & \text{if~ $2m-n\equiv0\pmod{3}$}
\end{cases}
\end{align*}
\end{description}
with equality if and only if the degree set of $G$ is \\
$\bullet$ $\{1,2,4\}$ and $G$ contains only one vertex of degree $2$ whenever\, $2m-n\equiv1\pmod{3}$;\\
$\bullet$ $\{1,3,4\}$ and $G$ contains only one vertex of degree $3$ whenever\, $2m-n\equiv2\pmod{3}$;\\
$\bullet$ $\{1,4\}$ whenever\, $2m-n\equiv0\pmod{3}$.
\end{theorem}

\begin{proof}
Because the proofs of the both parts are similar to each other, we prove only Part (i). If the inequality $n_2+n_3\ge2$ holds, then by using Lemma \ref{MT-Lem-1} and Equation \eqref{MT-Eq-7}, one has
\begin{align*}
H_f(G)&< \frac{1}{3}\Big(4f(1) - f(4)\Big)n + \frac{2}{3}\Big(f(4) - f(1)\Big)m\\[4mm]
& \quad \ + \min \left\{ f(2) -\frac{2}{3}f(1) -\frac{1}{3}f(4),  f(3) -\frac{1}{3}f(1) -\frac{2}{3}f(4)  \right\}\\[4mm]
&< \frac{1}{3}\Big(4f(1) - f(4)\Big)n + \frac{2}{3}\Big(f(4) - f(1)\Big)m
\end{align*}
as desired.\\
In the remaining proof, assume that $n_2+n_3\le1$.
Then, $(n_2,n_3)\in\{(0,0),(1,0),(0,1)\}$.
From Equations \eqref{MT-Eq-2} and  \eqref{MT-Eq-3}, it follows that $n_2+2n_3\equiv 2m-n\pmod{3}$ (see for example \cite{Hu}),
which gives
$$
(n_2,n_3)=
\begin{cases}
(1,0) & \text{if\,~ $2m-n\equiv1\pmod{3}$,}\\[2mm]
(0,1) & \text{if\,~ $2m-n\equiv2\pmod{3}$,}\\[2mm]
(0,0) & \text{if\,~ $2m-n\equiv0\pmod{3}$.}
\end{cases}
$$
The required result follows now from Equation \eqref{MT-Eq-5}.
\end{proof}

In what follows, we consider some well-known topological  indices that satisfy the assumptions of Theorem \ref{thm-main-01} and hence yield different corollaries of Theorem \ref{thm-main-01}.

First, we take $f(x)=x^\alpha$. Then $H_f$ is the general zeroth-order Randi\'c index $^0\!R_\alpha$. Here, we have
\[
\xi_1 = f(2) -\frac{2}{3}f(1) -\frac{1}{3}f(4)=
\begin{cases}
\displaystyle - \frac{(2^\alpha -2)(2^\alpha -1)}{3} <0 & \text{ if either $\alpha>1$ or $\alpha<0$,}\\[5mm]
\displaystyle - \frac{(2^\alpha -2)(2^\alpha -1)}{3} >0 & \text{ if $0<\alpha<1$,}
\end{cases}
\]
and
\[
\xi_2 = f(3) -\frac{1}{3}f(1) -\frac{2}{3}f(4)=
\begin{cases}
\displaystyle  \frac{3^{\alpha+1} - 2^{2\alpha+1} -1}{3} <0 & \text{ if either $\alpha>1$ or $\alpha<0$,}\\[5mm]
\displaystyle \frac{3^{\alpha+1} - 2^{2\alpha+1} -1}{3} >0 & \text{ if $0<\alpha<1$.}
\end{cases}
\]
Also,
\begin{equation}\label{frdo}
2\xi_2 = \frac{2(3^{\alpha+1} - 2^{2\alpha+1} -1)}{3}< \xi_1= - \frac{(2^\alpha -2)(2^\alpha -1)}{3} < \frac{\xi_2}{2}=\frac{3^{\alpha+1} - 2^{2\alpha+1} -1}{6}
\end{equation}
holds if either $\alpha>1$ or $\alpha<0$. If each inequality sign ``$<$'' of \eqref{frdo} is replaced with ``$>$'' then the resulting inequality holds for $0<\alpha<1$.
Thus, we have the following known \cite{Hu} result as a direct consequence of Theorem \ref{thm-main-01}.

\begin{corollary}\label{thm-main-cor-01}
Let\, $G$ be a chemical $(n,m)$-graph, where $n\ge5$. If either $\alpha>1$ or $\alpha<0$, then
\begin{align*}
^0\!R_\alpha(G)&\le \frac{4 - 4^\alpha}{3} \,n + \frac{2(4^\alpha-1)}{3}\, m  +
\begin{cases}
\displaystyle - \frac{(2^\alpha -2)(2^\alpha -1)}{3}       & \text{if\,~ $2m-n\equiv1\pmod{3}$}\\[4mm]
\displaystyle \frac{3^{\alpha+1} - 2^{2\alpha+1} -1}{3}    & \text{if\,~ $2m-n\equiv2\pmod{3}$}\\[4mm]
0                                                          & \text{if\,~ $2m-n\equiv0\pmod{3}$}
\end{cases}
\end{align*}
with equality if and only if the degree set of $G$ is the same as specified in Theorem \ref{thm-main-01}. If $0<\alpha<1$ then the above inequality for ~$^0\!R_\alpha(G)$ is reversed.

\end{corollary}

Now, we take $f(x)=xa^x$ with $a>0$ but $a\ne1$. Then $H_f$ is the variable sum exdeg index $SEI_a$. Here, we have
\[
\xi_1 = f(2) -\frac{2}{3}f(1) -\frac{1}{3}f(4)=
\begin{cases}
\displaystyle - \frac{2a(a-1) (2 a^2+2a-1) }{3}<0 & \text{ if either $0<a<\displaystyle\frac{1}{3}$ or $a>1$}\\[5mm]
\displaystyle - \frac{2a(a-1) (2 a^2+2a-1) }{3} >0 & \text{ if $\displaystyle\frac{1}{2} < a < 1$,}
\end{cases}
\]
and
\[
\xi_2 = f(3) -\frac{1}{3}f(1) -\frac{2}{3}f(4)=
\begin{cases}
\displaystyle - \frac{a(a-1) (8 a^2-a-1) }{3}<0 & \text{ if either $0<a<\displaystyle\frac{1}{3}$ or $a>1$}\\[5mm]
\displaystyle - \frac{a(a-1) (8 a^2-a-1) }{3} >0 & \text{ if $\displaystyle\frac{1}{2} < a < 1$,}
\end{cases}
\]
Also,
\begin{align}\label{frdo01}
2\xi_2 = -\frac{2a(a-1) (8 a^2-a-1) }{3}&< \xi_1= - \frac{2a(a-1) (2 a^2+2a-1) }{3}\nonumber\\[3mm]
&< \frac{\xi_2}{2}=-\frac{a(a-1) (8 a^2-a-1) }{6}
\end{align}
holds if either $a>1$ or $0<a<\frac{1}{3}$. If each inequality sign ``$<$'' in \eqref{frdo01} is replaced with ``$>$'' then the resulting inequality holds for $\frac{1}{2} < a < 1$.
Thus, we have the next result that follows directly from Theorem \ref{thm-main-01}.

\begin{corollary}\label{thm-main-cor-02}
Let\, $G$ be a chemical $(n,m)$-graph, where $n\ge5$. If either $a>1$ or $0<a<\frac{1}{3}$, then
{\small
\begin{align*}
SEI_a(G)&\le \frac{4a(1 - a^3)n}{3} + \frac{2a(4a^3-1)m}{3}  +
\begin{cases}
\displaystyle  \frac{2a(1-a) (2 a^2+2a-1) }{3}       & \text{if\,~ $2m-n\equiv1\hspace{-3mm}\pmod{3}$}\\[4mm]
\displaystyle  \frac{a(1-a) (8 a^2-a-1) }{3}         & \text{if\,~ $2m-n\equiv2\hspace{-3mm}\pmod{3}$}\\[4mm]
0                                                    & \text{if\,~ $2m-n\equiv0\hspace{-3mm}\pmod{3}$}
\end{cases}
\end{align*}
}
with equality if and only if the degree set of $G$ is the same as specified in Theorem \ref{thm-main-01}. If $\frac{1}{2} < a < 1$ then the above inequality for $SEI_a(G)$ is reversed.

\end{corollary}

Next, we take $f(x)=x(\ln x)^a$ with $a>0$. Then $H_f$ is the variable sum lodeg index $SLI_a$.
Here, for $a>\frac{\ln3-\ln4}{\ln(\ln2)-\ln(\ln3)} ~(\approx 0.6246)$, we have
\[
\xi_1 = f(2) -\frac{2}{3}f(1) -\frac{1}{3}f(4)=
\frac{2\big(3 (\ln2)^a-2 (\ln4)^a\big)}{3}<0,\]
\[
\xi_2 = f(3) -\frac{1}{3}f(1) -\frac{2}{3}f(4)=
\frac{9(\ln3)^a-8 (\ln4)^a}{3}<0
\]
and
\begin{align*}
2\xi_2 = \frac{2\big(9(\ln3)^a-8 (\ln4)^a\big)}{3}&< \xi_1= \frac{2\big(3 (\ln2)^a-2 (\ln4)^a\big)}{3}< \frac{\xi_2}{2}=\frac{9(\ln3)^a-8 (\ln4)^a}{6}.
\end{align*}
Hence, the following corollary is another direct consequence of Theorem \ref{thm-main-01}.

\begin{corollary}\label{thm-main-cor-03}
Let\, $G$ be a chemical $(n,m)$-graph, where $n\ge5$. If
\[
a>\frac{\ln3-\ln4}{\ln(\ln2)-\ln(\ln3)} ~~(\approx 0.6246),
\]
then
\begin{align*}
SLI_a(G)&\le \frac{8(\ln4)^a}{3}\,m - \frac{4(\ln4)^a}{3}\,n  +
\begin{cases}
\displaystyle  \frac{2\big(3 (\ln2)^a-2 (\ln4)^a\big)}{3}       & \text{if\,~ $2m-n\equiv1\hspace{-3mm}\pmod{3}$}\\[4mm]
\displaystyle  \frac{9(\ln3)^a-8 (\ln4)^a}{3}        & \text{if\,~ $2m-n\equiv2\hspace{-3mm}\pmod{3}$}\\[4mm]
0                                                          & \text{if\,~ $2m-n\equiv0\hspace{-3mm}\pmod{3}$}
\end{cases}
\end{align*}
with equality if and only if the degree set of $G$ is the same as specified in Theorem \ref{thm-main-01}.

\end{corollary}

Finally, if we take $f(x)=(n-1-x)x^2$, or $f(x)=\ln x^{ax}$, or $f(x)=\ln x^a$, then $H_f$ is the forgotten topological coindex $\overline{F}(G)$ (see \cite{c14,c26}),  or the natural logarithm of the general multiplicative first Zagreb index $\ln \Pi_{1,a}$, or the natural logarithm of the general multiplicative second Zagreb index $\ln \Pi_{2,a}$, respectively.
\begin{itemize}
  \item If we take $f(x)=(n-1-x)x^2$ with $n\ge11$, or $f(x)=\ln x^{ax}$ with $a>0$, or $f(x)=\ln x^a$ with $a<0$, then $f$ satisfies the conditions of Theorem \ref{thm-main-01}(i).
  \item If we take $f(x)=\ln x^a$ with $a>0$, or $f(x)=\ln x^{ax}$ with $a<0$, then $f$ satisfies the conditions of Theorem \ref{thm-main-01}(ii).
\end{itemize}
Hence, the next result follows immediately from Theorem \ref{thm-main-01}.

\begin{corollary}\label{thm-main-cor-04}
Let\, $G$ be a chemical $(n,m)$-graph, where $n\ge5$. If $a<0$ then
\begin{align*}
\Pi_{1,a}(G)&\le
\begin{cases}
\displaystyle  2^{ \frac{a (4 m - 2 n + 1)}{3}}       & \text{if\,~ $2m-n\equiv1\pmod{3}$}\\[4mm]
\displaystyle  2^{ \frac{2a (2 m -  n - 2)}{3}}3^a        & \text{if\,~ $2m-n\equiv2\pmod{3}$}\\[4mm]
\displaystyle  2^{ \frac{2a (2 m -  n )}{3}}                                                        & \text{if\,~ $2m-n\equiv0\pmod{3}$,}
\end{cases}
\end{align*}
\begin{align*}
\Pi_{2,a}(G)&\ge
\begin{cases}
\displaystyle   2^{ \frac{2a (8 m - 4 n - 1)}{3}}       & \text{if\,~ $2m-n\equiv1\pmod{3}$}\\[4mm]
\displaystyle   2^{ \frac{8a (2 m -  n - 2)}{3}}3^{3a}    & \text{if\,~ $2m-n\equiv2\pmod{3}$}\\[4mm]
\displaystyle   2^{ \frac{8a (2 m - n )}{3}}                                                         & \text{if\,~ $2m-n\equiv0\pmod{3}$,}
\end{cases}
\end{align*}
and if $n\ge11$ then
\[
\overline{F}(G)\le
\begin{cases}
2 \Big(m ( 5 n - 26) - n(2n-11)  + 8\Big)       & \text{if\,~ $2m-n\equiv1\pmod{3}$}\\[4mm]
2 \Big(m ( 5 n - 26) - n(2n-11)  + 9\Big)        & \text{if\,~ $2m-n\equiv2\pmod{3}$}\\[4mm]
2\Big(m ( 5 n - 26) - 2n(n-6)\Big)                                                & \text{if\,~ $2m-n\equiv0\pmod{3}$,}
\end{cases}
\]
where the equality sign in any of the above inequalities involving $\Pi_{1,a}(G)$, $\Pi_{2,a}(G)$, $\overline{F}(G)$, holds if and only if the degree set of $G$ is the same as specified in Theorem \ref{thm-main-01}. If $a>0$ then the above inequalities involving $\Pi_{1,a}(G)$ and $\Pi_{2,a}(G)$ are reversed.
\end{corollary}

From Theorem \ref{thm-main-01} and the identity \eqref{l10a}, the next result follows.

\begin{theorem}
  \label{t3}
  Let $G$ be a chemical $(n,m)$-graph, where $n\ge5$. Let $\xi_1$ and $\xi_2$ be the numbers defined in \eqref{MT-Eq-rev-1}.
  \begin{description}
   \item[(i)]  If both $\xi_1$ and $\xi_2$ are negative such that $2\xi_2 <\xi_1 < \xi_2/2$, then
 \begin{align*}
TI(G)+\overline{TI}(G) &\le (n-1)\left( \frac{1}{3}\Big(4f(1) - f(4)\Big)n + \frac{2}{3}\Big(f(4) - f(1)\Big)m\right)\nonumber\\[4mm]
& \quad \ +
\begin{cases}
\displaystyle (n-1)\Big(f(2) -\frac{2}{3}f(1) -\frac{1}{3}f(4)\Big)       & \text{if~ $2m-n\equiv1\pmod{3}$}\\[6mm]
\displaystyle (n-1)\Big(f(3) -\frac{1}{3}f(1) -\frac{2}{3}f(4)\Big)       & \text{if~ $2m-n\equiv2\pmod{3}$}\\[6mm]
0                                                          & \text{if~ $2m-n\equiv0\pmod{3}$,}
\end{cases}
\end{align*}
with equality if and only if the degree set of $G$ is\\
$\bullet$ $\{1,2,4\}$ and $G$ contains only one vertex of degree $2$ whenever\, $2m-n\equiv1\pmod{3}$;\\
$\bullet$ $\{1,3,4\}$ and $G$ contains only one vertex of degree $3$ whenever\, $2m-n\equiv2\pmod{3}$;\\
$\bullet$ $\{1,4\}$ whenever\, $2m-n\equiv0\pmod{3}$.

\item [(ii)] If both $\xi_1$ and $\xi_2$ are positive such that $\xi_2/2 < \xi_1 < 2\xi_2$, then
\begin{align*}
TI(G)+\overline{TI}(G) &\ge (n-1)\left( \frac{1}{3}\Big(4f(1) - f(4)\Big)n + \frac{2}{3}\Big(f(4) - f(1)\Big)m\right)\nonumber\\[4mm]
& \quad \ +
\begin{cases}
\displaystyle (n-1)\Big(f(2) -\frac{2}{3}f(1) -\frac{1}{3}f(4)\Big)       & \text{if\,~ $2m-n\equiv1\pmod{3}$}\\[6mm]
\displaystyle (n-1)\Big(f(3) -\frac{1}{3}f(1) -\frac{2}{3}f(4)\Big)       & \text{if\,~ $2m-n\equiv2\pmod{3}$}\\[6mm]
0                                                          & \text{if\,~ $2m-n\equiv0\pmod{3}$}
\end{cases}
\end{align*}
with equality if and only if the degree set of $G$ is \\
$\bullet$ $\{1,2,4\}$ and $G$ contains only one vertex of degree $2$ whenever\, $2m-n\equiv1\pmod{3}$;\\
$\bullet$ $\{1,3,4\}$ and $G$ contains only one vertex of degree $3$ whenever\, $2m-n\equiv2\pmod{3}$;\\
$\bullet$ $\{1,4\}$ whenever\, $2m-n\equiv0\pmod{3}$.
\end{description}
\end{theorem}

\section*{Acknowledgements}
The authors are very grateful to the handling editor and the referee for  helpful comments and suggestions.  This research has been funded by the Scientific Research Deanship, University of Ha\!'il, Saudi Arabia, through project number
RG-22\,005.

\end{document}